\documentclass[letter, 10pt, conference]{ieeeconf}
\IEEEoverridecommandlockouts        % This command is only
\usepackage{etex}
\usepackage[vlined,ruled]{algorithm2e}
\usepackage[dvipsnames]{xcolor}
\usepackage{pgfpages}
\usepackage[cmex10]{amsmath}		%Load Fonts
\usepackage{amssymb, mathrsfs}
\usepackage{cite}
\usepackage{url}

\usepackage{dsfont}
\usepackage{siunitx}
\usepackage{verbatim}									%verbatim and comment environemtns		
\usepackage{arydshln}
\usepackage{mathtools}									%used for xmapsto command
\usepackage{siunitx}									%load SI units
\usepackage{mdframed}

\usepackage{graphicx}
\usepackage{epstopdf}

\usepackage[font=small]{caption}
\usepackage{subcaption}
\usepackage[activate={true,nocompatibility},final,tracking=true,kerning=true,spacing=true,factor=1100,stretch=10,shrink=10]{microtype}

\usepackage{array}
\usepackage{multirow}
\usepackage{enumerate}
\usepackage{booktabs}

\graphicspath{{./fig/}}

\usepackage{pgfplots}
\pgfplotsset{width=8cm,compat=newest}
\newcommand*{\damping}{0.006}%
\newcommand*{\freq}{25}%
\pgfmathsetmacro{\freqd}{sqrt(1-(\damping)^2)*\freq}%

\pgfplotsset{
    standard/.style={
    axis x line=middle,
    axis y line=middle,
    enlarge x limits=0.15,
	enlarge y limits=0.15,
	every axis plot post/.style={mark options={fill=black}},
	}
}

\pgfplotsset{%
    ,compat=1.12
    ,every axis x label/.style={at={(current axis.right of origin)},anchor=north west}
    ,every axis y label/.style={at={(current axis.above origin)},anchor=north east}
    }

\usepackage{tikz}
\usetikzlibrary{arrows}
\usetikzlibrary{decorations.pathmorphing}
\usetikzlibrary{decorations.markings}
\usetikzlibrary{snakes}
\usetikzlibrary{matrix}
\usetikzlibrary{calc}
\usetikzlibrary{patterns}
\usetikzlibrary{positioning}
\usetikzlibrary{shapes}
\usetikzlibrary{shapes.symbols}
\usetikzlibrary{shapes.geometric}
\usetikzlibrary{shadows.blur}
\usetikzlibrary{shapes.arrows}
\usetikzlibrary{shapes.multipart}
\usetikzlibrary{shapes.callouts}
\usetikzlibrary{shapes.misc}

\tikzstyle{every node}=[font=\small]
\tikzstyle{every path}=[line width=0.8pt,line cap=round,line join=round]

\usepackage[american,smartlabels]{circuitikz}
\ctikzset{bipoles/thickness=1}
\ctikzset{bipoles/length=0.8cm}
\ctikzset{bipoles/diode/height=.375}
\ctikzset{bipoles/diode/width=.3}
\ctikzset{tripoles/thyristor/height=.8}
\ctikzset{tripoles/thyristor/width=1}
\ctikzset{bipoles/vsourceam/height/.initial=.7}
\ctikzset{bipoles/vsourceam/width/.initial=.7}

\newcommand{\real}{\mathbb{R}}

\newcommand{\setdef}[2]{\{#1 \;|\; #2\}}

\DeclareMathOperator*{\minimize}{minimize} 									% minimize
\newcommand{\vect}[1]{\mathbbold{#1}}

\newcommand{\vzeros}[1][]{\vect{0}_{#1}}
\DeclareSymbolFont{bbold}{U}{bbold}{m}{n}
\DeclareSymbolFontAlphabet{\mathbbold}{bbold}

\newcommand{\map}[3]{#1: #2 \rightarrow #3}

\renewcommand{\theenumi}{(\roman{enumi})}

\usepackage{soul}

\DeclareMathOperator{\nullspace}{null}
\DeclareMathOperator{\range}{range}
\DeclareMathOperator{\rank}{rank}
\newcommand{\define}{\coloneqq}

\DeclareMathOperator{\subto}{subject~to}

\newcommand{\gradmat}{Q}

\newcommand\xqed[1]{%
  \leavevmode\unskip\penalty9999 \hbox{}\nobreak\hfill
  \quad\hbox{#1}}
\newcommand\envend{\xqed{$\triangle$}}

\newcommand\oprocendsymbol{\hbox{$\square$}}
\newcommand\oprocend{\relax\ifmmode\else\unskip\hfill\fi\oprocendsymbol}

\newtheorem{theorem}{Theorem}[section]

\newtheorem{proposition}[theorem]{Proposition}
\newtheorem{problem}{Problem}[section]

\newtheorem{remark}{Remark}[section]
\newtheorem{assumption}{Assumption}[section]

\newif\ifforstudents

\binoppenalty=\maxdimen
\relpenalty=\maxdimen

\title{\bf Optimal Steady-State Control for Linear Time-Invariant Systems
\thanks{
}
}

\author{Liam S. P. Lawrence, Zachary E. Nelson, Enrique Mallada, and John W. Simpson-Porco%
    \thanks{
      L. S. P. Lawrence and J. W. Simpson-Porco are with the Department of Electrical and Computer Engineering, University of Waterloo, ON, Canada. Email: {\tt \{lsplawre,jwsimpson\}@uwaterloo.ca}. Z. E. Nelson and E. Mallada are with the Department of Electrical and Computer Engineering, Johns Hopkins University, 3400 N. Charles Street, Baltimore, MD 21218. Email: {\tt \{znelson2,mallada\}@jhu.edu}. L. S. P. Lawrence and J. W. Simpson-Porco acknowledge the support of NSERC Discovery Grant RGPIN-2017-04008 and the CGS M program. The work of E. Mallada and Z. E. Nelson was supported by ARO through contract W911NF-17-1-0092, and NSF through grants CNS 1544771 and CAREER 1752362. %This work was supported in part by the NSERC Discovery Grant  RGPIN-2017-04008 and by University of Waterloo start-up funding.
}
}

\begin{document}
\maketitle
\thispagestyle{empty}
\pagestyle{empty}

%%%%%%%%%%%%%%%%%%%%%%%%%%%%%%%%%%%%%%%%%%%%%%%%%%%%%%%%%%%%%%%%%%%%%%%%%%%%%%%%

\begin{abstract}
  We consider the problem of designing a feedback controller that guides the input and output of a linear time-invariant system to a minimizer of a convex optimization problem. The system is subject to an unknown disturbance that determines the feasible set defined by the system equilibrium constraints. Our proposed design enforces the Karush-Kuhn-Tucker optimality conditions in steady-state without incorporating dual variables into the controller. We prove that the input and output variables achieve optimality in equilibrium and outline two procedures for designing controllers that stabilize the closed-loop system. We explore key ideas through simple examples and simulations.
\end{abstract}

% -------------------------------------------------------------------------------------- %
\section{Introduction}\label{Sec:Introduction}

%{\tb
%\textbf{John's Open Questions:}
%\begin{enumerate}
%%\item Comparison of proposed approach to dynamic dual variable approach in some I/O performance metric ($\mathcal{H}_2,\mathcal{H}_{\infty}$)
%\end{enumerate}
%}

%{\tb
%\textbf{Enriques Open Questions:}
%\textbf{Short Term:}
%\begin{enumerate}
%\item Present the results as a general framework: Some of the optimality and stability theorems can be written independently of the "algorithm/controller" you chose. Thus whether we talk about OSS, or our ACC version, the theorem will still hold. This will essentially make sure that other solutions based on the same framework will have to definitely cite us. However, you are much better at getting people's attention, so your comments on this are very welcome...
%\item Find a structure that guarantees the feasibility of the IQCs: This is related to Mihailo's work. That is can we find sufficient conditions on the parameters that guarantee feasibility? Maybe through small gain, or passivity? The latter may be possible at least for OSS given that it doesn't have an observer.
%\end{enumerate}
%\textbf{Long Term:}
%\begin{enumerate}
%\item Inequality constraints
%\item Perfect tracking, i.e., internal model principle type of results...
%\item Nonlinear plants
%\end{enumerate}
%}

Many engineering systems must be operated at an ``optimal'' steady-state that minimizes operational costs. For example, the generators that supply power to the electrical grid are scheduled according to the solution of an optimization problem which minimizes the total production cost of the generated power \cite{AKB-AA-TS:14}. This same theme of guiding system variables to optimizers emerges in other areas, such as network congestion management \cite{FPK-AKM-DKHT:98,EM-FP:08}, chemical processing \cite{DD-MG:05,MG-DD-MP:05}, and wind turbine power capture \cite[Section 2.7]{IM-AIB-NAC-EC:08},\cite{BB-HS:10}.

Traditionally, the optimal steady-state set-points are computed offline in advance, and then controllers are used in real time to track the set-points. However, this two-step design method is inefficient if the set-points must be updated repeatedly and often. For instance, the increased use of renewable energy sources causes rapid fluctuations in power networks, which decreases the effectiveness of the separated approach due to the rapidly-changing optimal steady-state \cite{TS-CDP-AvdS:17}. Furthermore, the two-step method is infeasible if \emph{unmeasurable disturbances} change the optimal set-points. To continuously keep operational costs to a minimum, such systems should instead employ a controller that continuously solves the optimization problem and guides the system to an optimizer despite disturbances; we will call the problem of designing such a controller the \emph{optimal steady-state} (OSS) control problem.

The prevalence of the OSS control problem in applications has motivated much work on its solution for general system classes. In the \emph{extremum-seeking control} approach to OSS control, a harmonic signal is used to perturb an uncertain system, and the gradient of a cost function is then inferred by filtering system measurements; a control signal is applied to drive the gradient to zero \cite{KM-HHW:00,DD-MG:05,MG-EM-DD:15}. Joki{\'c}, Lazar, and van den Bosch propose using the Karush-Kuhn-Tucker conditions for optimality as the basis of a nonlinear feedback controller that guides the outputs of a system to an optimizer \cite{AJ-ML-PPJvdB:09b,AJ-ML-PPJvdB:08}. Nelson and Mallada consider an optimization problem over system states and apply gradient feedback with a proportional-integral (PI) controller; if the full system state cannot be directly measured, a Luenberger observer is employed \cite{ZEN-EM:18}.

Much of the literature on OSS control problems focuses on the optimization of either the steady-state input \emph{or} the steady-state output of the system. The optimal power flow problem, for example, concerns the minimization of the total cost of power produced by dispatchable units that serve as inputs \cite{FD-JWSP-FB:13y,LG-SHL:16,EM-CZ-SL:17}. Other work has focused on optimization problems over the output only; the cost of applying the input necessary to achieve the desired output is not considered \cite{EDA-SVD-GBG:15,FDB-HBD-CE:12,DD-MG:05,KH-JPH-KU:14,AJ-ML-PPJvdB:09b,AJ-ML-PPJvdB:08}. 

These input-only or output-only optimization designs are formulated under the assumption that the resulting optimizer is in fact \emph{consistent} with the dynamic system operating in equilibrium. As such, the key obstacle to relaxing this assumption \textemdash{} and to extending the formulations to steady-state input-output optimization \textemdash{} is that an arbitrary input-output optimizer \emph{need not} be consistent with dynamic equilibrium. Controllers proposed currently in the literature for optimization over both the input and output only apply to systems with special structure: \cite{XZ-AP-NL:18,EM-CZ-SL:17} assume the plant can be interpreted as a primal-dual gradient algorithm, while Hatanaka and colleagues consider the specific case of temperature dynamics in a building\cite{TH-XZ-WS-MZ-NL:17}. An OSS controller design for the general case is still lacking.

\subsection{Contributions}
The main contribution of this paper is an OSS controller design for any stabilizable linear time-invariant (LTI) system where the steady-state objective function depends on both the input and output. We ensure the optimizer is consistent with the equilibrium conditions for the LTI system by including the equilibrium equations as explicit equality constraints of the optimization problem. Our key insight is that, rather than enforcing these equality constraints with dual variables in the controller, one may enforce that the gradient of the objective function lie (at equilibrium) in a subspace defined by the LTI system matrices. This eliminates the need for additional dynamic controller states, resulting in a lower-dimensional controller design. We offer two strategies to ensure closed-loop stability: search for appropriate proportional-integral gains using a linear matrix inequality stability criterion, or design a dynamic stabilizer using $\mathcal{H}_\infty$ synthesis methods.

\subsection{Notation}
The $n \times n$ identity matrix is $I_n$, $\vzeros[]$ is a matrix of zeros of appropriate dimension, and $\vzeros[n]$ is the $n$-vector of all zeros. The vector norm $\|\cdot\|$ is always assumed to be the Euclidean norm $\|\cdot\|_2$. For symmetric matrices $A$ and $B$, the inequality $A \succ B$ means $A-B$ is positive definite, while $A \succeq B$ means $A-B$ is positive semidefinite. If $\map{h}{\real^n}{\real}$ is differentiable, $\map{\nabla h}{\real^n}{\real^n}$ denotes its gradient. We say $h: \real^n \to \real$ is \emph{strongly convex} with parameter $\kappa > 0$ if $\left(\nabla h(x) - \nabla h(y)\right)^{\sf T}(x-y) \geq \kappa\|x-y\|^2$ for all $x,y \in \real^n$. The gradient $\nabla h$ is \emph{(globally) Lipschitz continuous} with parameter $L > 0$ if $\|\nabla h(x) - \nabla h(y)\| \leq L\|x-y\|$ for all $x,y \in \real^n$.

\section{OSS Control Problem Statement}\label{Sec:ProblemStatement}
We consider the LTI system
\begin{subequations}\label{Eq:LtiSystem}
  \begin{align}
    \dot{x} &= Ax+Bu+d\\
    y &= Cx
  \end{align}
\end{subequations}
as the plant whose input and output we wish to optimize in steady-state. We omit the direct throughput term $Du$ for simplicity of presentation. The control input is $u \in \real^m$, the state is $x \in \real^n$, and the output is $y \in \real^p$. We assume that $m, p \leq n$. The plant is subject to an \emph{unknown but constant} disturbance $d \in \real^{n}$. We denote equilibrium values of the state, input, and output by $\bar{x}$, $\bar{u}$, and $\bar{y}$ respectively, which satisfy $A\bar{x}+B\bar{u}+d=\vzeros[n]$ and $\bar{y} = C\bar{x}$.

The desired steady-state operating point $(\bar{y},\bar{u})$ of the system \eqref{Eq:LtiSystem} is determined by the optimization problem
\begin{equation}\label{Eq:OptimizationProblem}
  \begin{aligned}
    \minimize_{\bar{x},\,\bar{y},\,\bar{u}} &\quad g(\bar{y},\bar{u})\\
    \subto &\quad A\bar{x}+B\bar{u}+d=\vzeros[n]\\
    &\quad \bar{y} = C\bar{x}\,.
  \end{aligned}
\end{equation}
The cost function $g:\real^p \times \real^m \to \real$ is a differentiable convex function of the equilibrium output $\bar{y}$ and the equilibrium control input $\bar{u}$. The feasible region is the set of forced equilibrium points for the plant dynamics. Including additional equality and inequality constraints in \eqref{Eq:OptimizationProblem} is a relatively straightforward extension of the results that follow, and will be detailed in a forthcoming extended manuscript.

Consider a nonlinear feedback controller for the plant \eqref{Eq:LtiSystem} with state $x_{\rm c} \in \real^{n_{\rm c}}$ of the form
\begin{equation}\label{Eq:NonlinearController}
  \begin{aligned}
    \dot{x}_{\rm c} &= F_{\rm c}(x_{\rm c},y,u)\\
    u &= H_{\rm c}(x_{\rm c},y,u)\,.
  \end{aligned}
\end{equation}
The precise problem we wish to solve is as follows.

\smallskip

\begin{problem}[\bf Optimal Steady-State Control Problem]\label{Prob:Oss}
  Design a feedback controller of the form \eqref{Eq:NonlinearController} processing measurements $y(t)$ and producing control signal $u(t)$ such that for any constant disturbance vector $d \in \real^n$ the following hold for the closed-loop system \eqref{Eq:LtiSystem} and \eqref{Eq:NonlinearController}:
  \begin{enumerate}[(i)]
  \item the closed-loop system has a globally asymptotically stable equilibrium point;
  \item for any initial condition $(x(0),x_{\rm c}(0)) \in \real^n \times \real^{n_{\rm c}}$, $\lim\limits_{t\rightarrow \infty} (y(t),u(t))$ is an optimizer of \eqref{Eq:OptimizationProblem}.\envend
  \end{enumerate}
\end{problem}

\smallskip

As the control system operates in real time, the disturbance $d$ may change, and each change in its value changes the optimal operating point of \eqref{Eq:OptimizationProblem} by modifying the feasible set. Problem \ref{Prob:Oss} specifies that the controller must maintain stability and automatically guide the plant to the new optimizer.

We make a number of assumptions that are necessary for the OSS control problem to be feasible.

\smallskip

\begin{assumption}[\bf Stabilizable Plant]\label{Assump:Stabilizable}
  The matrix pair $(A,B)$ is stabilizable.\envend
\end{assumption}

\smallskip

By the Popov-Belevitch-Hautus test, $(A,B)$ is stabilizable if and only if $\rank\begin{bmatrix}A-\lambda I_n&B\end{bmatrix} = n$ for all complex numbers $\lambda$ with nonnegative real part \cite[Section 14.3]{JPH:09}. The OSS control problem is infeasible if the plant fails to be stabilizable, as a lack of internal stability precludes reaching steady-state from an arbitrary initial condition.

\smallskip

\begin{remark}[\bf Nonempty Feasible Set]\label{Rem:NonemptyFeasibleSet}
Assumption \ref{Assump:Stabilizable} implies the feasible set of \eqref{Eq:OptimizationProblem} is non-empty. Take $\lambda = 0$ in the definition of stabilizability to obtain $\rank\begin{bmatrix}A&B\end{bmatrix} = n$. It follows that a solution $(\bar{x},\bar{u})$ of $A\bar{x} + B\bar{u} + d = \vzeros[n]$ exists for any $d \in \real^n$. \envend
\end{remark}

\smallskip

\begin{assumption}[\bf Optimizer Exists]\label{Assump:SolutionExists}
  For each $d \in \real^n$, an optimizer of \eqref{Eq:OptimizationProblem} exists.\envend
\end{assumption}

\smallskip

The OSS control problem is infeasible if the optimization problem has no solution, since the desired steady-state operating point does not exist.

With these necessary assumptions in place, we move on to controller design. The OSS control problem is composed of two sub-problems. The first sub-problem is to design a controller that establishes the optimizers of \eqref{Eq:OptimizationProblem} as the only equilibrium points of the closed-loop system. The second sub-problem is to ensure closed-loop stability. In Section \ref{Sec:ControllerDesign} we present a controller that satisfies the equilibrium criterion, then in Section \ref{Sec:Stabilization} we discuss stabilization.

\section{OSS Controller Design}\label{Sec:ControllerDesign}
We will first examine the \emph{full-state measurement case} when $y = x$, and then show how to extend the results to the \emph{measurement case} when $y = Cx$. To distinguish the former from the latter, we denote the cost function in the optimization problem by $f(\bar{x},\bar{u})$ instead of $g(\bar{y},\bar{u})$. The desired steady-state operating point is determined by
\begin{equation}\label{Eq:OptimizationProblemFullInformation}
\begin{aligned}
  \minimize_{\bar{x},\bar{u}} &\quad f(\bar{x},\bar{u})\\
  \subto &\quad A\bar{x}+B\bar{u}+d=\vzeros[n]\,.
\end{aligned}
\end{equation}
We begin by discussing the optimality conditions associated with the problem \eqref{Eq:OptimizationProblemFullInformation}.

\subsection{Subspace Formulation of Karush-Kuhn-Tucker Conditions}\label{Sec:KktConditions}
For a convex optimization problem, the Karush-Kuhn-Tucker (KKT) conditions are a set of necessary and sufficient conditions for a point to be a global optimizer, provided \emph{strong duality} holds. For a convex optimization problem with affine constraints only \textemdash{} like problems \eqref{Eq:OptimizationProblem} and \eqref{Eq:OptimizationProblemFullInformation} \textemdash{} strong duality holds as long as the feasible set is nonempty and the domain of the objective function is open \cite[Section 5.2.3]{SB-LV:04}, both of which follow from our assumptions.

The KKT conditions for \eqref{Eq:OptimizationProblemFullInformation} state: $(x^\star,u^\star)$ is a global optimizer if and only if there exists a $\lambda^\star \in \real^n$ such that
\begin{subequations}\label{Eq:KktOriginal}
  \begin{align}
    Ax^\star+Bu^\star+d &= \vzeros[n]\label{SubEq:KktFeasibility}\\
    \nabla f(x^\star,u^\star) + \begin{bmatrix}A&B\end{bmatrix}^{\sf T}\lambda^\star &= \vzeros[n+m]\,.\label{SubEq:KktGradient}
  \end{align}
\end{subequations}
The first condition \eqref{SubEq:KktFeasibility} is simply feasibility of $(x^\star,u^\star)$. The second condition \eqref{SubEq:KktGradient} is the \emph{gradient condition}, with $\lambda^\star$ the vector of \emph{dual variables} associated with the constraints of the optimization problem.

We can achieve steady-state optimality by enforcing the KKT conditions \eqref{Eq:KktOriginal} in equilibrium using a feedback controller. Previous designs in the literature enforce constraints using controller states that act as dual variables; consider \cite[Equation (4a)]{AJ-ML-PPJvdB:09b}, or \cite[Equation (8f)]{XZ-AP-NL:18}. However, if the closed-loop system is internally stable, the plant dynamics already enforce the constraint \eqref{SubEq:KktFeasibility} in steady-state; this suggests that dynamic dual variables are unnecessary. We can dispense with dual variables entirely and simplify the design procedure with an alternative interpretation of the KKT conditions.

The gradient condition \eqref{SubEq:KktGradient} is equivalent to the statement $\nabla f(x^\star,u^\star) \in \range \begin{bmatrix}A&B\end{bmatrix}^{\sf T}$. By a fundamental theorem from linear algebra, $\nabla f(x^\star,u^\star) \in \range \begin{bmatrix}A&B\end{bmatrix}^{\sf T}$ if and only if $\nabla f(x^\star,u^\star) \in \left(\nullspace \begin{bmatrix}A&B\end{bmatrix}\right)^{\perp}$, where the superscript $\perp$ denotes the orthogonal complement \cite[Theorem 7.7]{SA:15}. Hence, the KKT conditions \eqref{Eq:KktOriginal} can be equivalently written: $(x^\star,u^\star)$ is a global optimizer if and only if
\begin{subequations}\label{Eq:KktAlternative}
  \begin{align}
    Ax^\star+Bu^\star+d &= \vzeros[n]\label{SubEq:KktFeasibilityAlternative}\\
    \nabla f(x^\star,u^\star) &\perp \nullspace \begin{bmatrix}A&B\end{bmatrix}\,.\label{SubEq:KktGradientAlternative}
  \end{align}
\end{subequations}
The condition \eqref{SubEq:KktGradientAlternative} can be interpreted geometrically as saying that the gradient $\nabla f(x^\star,u^\star)$ is orthogonal to the subspace of first-order feasible variations at an optimal point; see \cite[Section 3.1]{DPB:95b}. This alternative formulation of the KKT conditions does not explicitly require dual variables, and we can enforce \eqref{SubEq:KktGradientAlternative} in equilibrium using a well-chosen matrix in the controller, as we describe next.

\subsection{Controller Equations in the Full-State Measurement Case}\label{Sec:ControllerEquationsFullInformation}
To enforce \eqref{SubEq:KktGradientAlternative} in equilibrium, the controller must drive the component of $\nabla f$ in the subspace $\nullspace \begin{bmatrix}A&B\end{bmatrix}$ to zero. The core idea of our control strategy is to make this component of $\nabla f$ the error signal input to a servo-compensator \cite{ED:76} with integral action; here we will simply use a proportional-integral controller.

Since $\rank\begin{bmatrix}A&B\end{bmatrix}=n$ by Assumption \ref{Assump:Stabilizable}, it follows from the rank-nullity theorem that $\dim \nullspace \begin{bmatrix}A&B\end{bmatrix} = m$. Let $\{v_1,v_2,\ldots,v_m\}$ be a basis of $\nullspace \begin{bmatrix}A&B\end{bmatrix}$ and concatenate these vectors to form the matrix $\gradmat \in \real^{m \times (n+m)}$ as $\gradmat \define \begin{bmatrix}v_1&v_2&\cdots&v_m\end{bmatrix}$. We then have
\begin{equation}\label{Eq:GradMatrix}
  \nullspace \gradmat^{\sf T} = \left(\nullspace \begin{bmatrix}A&B\end{bmatrix}\right)^\perp\,.
\end{equation}
We take $e \define -\gradmat^{\sf T}\nabla f$ as the error signal. We can interpret $\gradmat^{\sf T}$ as extracting the component of $\nabla f$ in the subspace $\nullspace \begin{bmatrix}A&B\end{bmatrix}$, which must be driven to zero to achieve optimality.

The equations of the system in feedback with the controller in the full-state measurement case are
\begin{subequations}\label{Eq:Controller}
  \begin{align}
    \dot{x} &= Ax+Bu+d \label{SubEq:PlantDynamics}\\
    e &= -\gradmat^{\sf T}\nabla f(x,u) \label{SubEq:GradientError}\\
    \dot{\eta} &= e \label{SubEq:Integrator}\\
    u &= K_I\eta + K_Pe\,. \label{SubEq:Input}
  \end{align}
\end{subequations}
A block diagram for the closed-loop system is shown in Figure \ref{Fig:OssController}. Equation \eqref{SubEq:GradientError} is the error term which the controller brings to zero using proportional-integral control. The integrator is given by \eqref{SubEq:Integrator}. The input equation \eqref{SubEq:Input} contains proportional and integral gains $K_P,K_I \in \real^{m \times m}$, where $K_I$ is assumed to be invertible. We assume going forward that the closed-loop system is well-posed, meaning that for any $(x, \eta)$, the equations \eqref{SubEq:GradientError} and \eqref{SubEq:Input} can always be solved for a unique value of $u$.

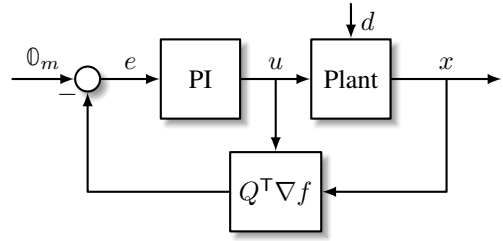
\begin{figure}[t]
\begin{center}
\tikzstyle{block} = [draw, fill=white, rectangle, 
    minimum height=3em, minimum width=6em, blur shadow={shadow blur steps=5}]
    \tikzstyle{hold} = [draw, fill=white, rectangle, 
    minimum height=3em, minimum width=3em, blur shadow={shadow blur steps=5}]
    \tikzstyle{dzblock} = [draw, fill=white, rectangle, minimum height=3em, minimum width=4em, blur shadow={shadow blur steps=5},
path picture = {
\draw[thin, black] ([yshift=-0.1cm]path picture bounding box.north) -- ([yshift=0.1cm]path picture bounding box.south);
\draw[thin, black] ([xshift=-0.1cm]path picture bounding box.east) -- ([xshift=0.1cm]path picture bounding box.west);
\draw[very thick, black] ([xshift=-0.5cm]path picture bounding box.east) -- ([xshift=0.5cm]path picture bounding box.west);
\draw[very thick, black] ([xshift=-0.5cm]path picture bounding box.east) -- ([xshift=-0.1cm, yshift=+0.4cm]path picture bounding box.east);
\draw[very thick, black] ([xshift=+0.5cm]path picture bounding box.west) -- ([xshift=+0.1cm, yshift=-0.4cm]path picture bounding box.west);
}]
\tikzstyle{sum} = [draw, fill=white, circle, node distance=1cm, blur shadow={shadow blur steps=8}]
\tikzstyle{input} = [coordinate]
\tikzstyle{output} = [coordinate]
\tikzstyle{split} = [coordinate]
\tikzstyle{pinstyle} = [pin edge={to-,thin,black}]

\begin{tikzpicture}[auto, scale = 0.6, node distance=2cm,>=latex', every node/.style={scale=1}]
  \node [input, name=input] {};
  \node [sum, right of=input] (sum) {};
  \node [hold, right of=sum, node distance=1.5cm] (pi) {PI};
  \node [hold, right of=pi, node distance=2cm] (plant) {Plant};
  \node [split, right of=pi, node distance=1cm] (split) {};
  \node [input, above of=plant, node distance=1cm] (disturbance) {};
  \node [hold, below of=split, node distance=1.5cm] (grad)  {$\gradmat^{\sf T}\nabla f$};
  \node [output, right of=plant, node distance=2cm] (output) {};
  
  Once the nodes are placed, connecting them is easy. 
  \draw [draw, -latex] (input) -- node {$\vzeros[m]$} (sum);
  \draw [draw, -latex] (sum) -- node {$e$} (pi);
  \draw [draw, -latex] (pi) -- node [name=u] {$u$} (plant);
  \draw [draw, -latex] (plant) -- node [name=y] {$x$} (output);
  \draw [draw, -latex] (y) |- (grad);
  \draw [draw, -latex] (grad) -| (sum) node[pos=0.99] {$-$};
  \draw [draw, -latex] (disturbance) -- node {$d$} (plant);
  \draw [draw, -latex] (split) -- (grad);
\end{tikzpicture}
\end{center}
\caption{Block diagram of the plant and a feedback controller that solves the OSS control problem in the full-state measurement case.}
\label{Fig:OssController}
\end{figure}

We now show the closed-loop system achieves optimality in steady-state. For a constant disturbance $d$, the equilibrium points $(\bar{x},\bar{\eta})$ of the closed-loop system \eqref{Eq:Controller} are the solutions of
\begin{subequations}\label{Eq:ControllerEquilibrium}
  \begin{align}
    \vzeros[n] &= A\bar{x}+B\bar{u}+d \label{SubEq:PlantDynamicsEquilibrium}\\
    \vzeros[m] &= \gradmat^{\sf T}\nabla f(\bar{x},\bar{u}) \label{SubEq:GradientErrorEquilibrium}\\
    \bar{u} &= K_I\bar{\eta}\,.
  \end{align}
\end{subequations}
Define the set of equilibrium state-input pairs as
\begin{equation}\label{Eq:EquilibriumSet}
  \mathcal{E} \define  \setdef{(\bar{x},\bar{u})}{(\bar{x},\bar{\eta}) = (\bar{x},K_I^{-1}\bar{u}) \,\,\text{solves \eqref{Eq:ControllerEquilibrium}}}
\end{equation}
and let the set of optimizers for the optimization problem be
\begin{equation}\label{Eq:OptimizerSet}
  \chi \define \setdef{(x^\star,u^\star) \in \real^n \times \real^m}{(x^\star,u^\star)\,\, \text{solves \eqref{Eq:OptimizationProblemFullInformation}}}\,.
\end{equation}

\smallskip

\begin{theorem}[\bf Equilibria and Optimizers]\label{Thm:EquilibriaAndOptimizers}
  The sets $\mathcal{E}$ and $\chi$ of \eqref{Eq:EquilibriumSet} and \eqref{Eq:OptimizerSet} satisfy $\mathcal{E} = \chi$.
\end{theorem}

\begin{proof}
  First note that \eqref{SubEq:PlantDynamicsEquilibrium} and \eqref{SubEq:GradientErrorEquilibrium} are equivalent to the KKT conditions \eqref{Eq:KktAlternative}. The former is the feasibility condition \eqref{SubEq:KktFeasibilityAlternative}, and the latter implies \eqref{SubEq:KktGradientAlternative} since
   \begin{equation*}
    \nabla f(\bar{x},\bar{u}) \in \nullspace \gradmat^{\sf T} = \left(\nullspace \begin{bmatrix}A&B\end{bmatrix}\right)^\perp\,.
  \end{equation*}
  We show that $\chi \subset \mathcal{E}$. The set of optimizers $\chi$ is non-empty by Assumption \ref{Assump:SolutionExists}, so there exists a $(x^\star,u^\star) \in \chi$. Since the KKT conditions are necessary for optimality, $(x^\star,u^\star)$ satisfy \eqref{SubEq:PlantDynamicsEquilibrium} and \eqref{SubEq:GradientErrorEquilibrium}. Therefore, $(\bar{x},\bar{\eta}) = (x^\star,K_I^{-1}u^\star)$ is a solution of \eqref{Eq:ControllerEquilibrium}. It follows that $\chi \subset \mathcal{E}$, which also implies $\mathcal{E}$ is non-empty. Conversely, we now show $\mathcal{E} \subset \chi$. If $(\bar{x},\bar{u}) \in \mathcal{E}$, then $\bar{x}$ and $\bar{u}$ satisfy \eqref{SubEq:PlantDynamicsEquilibrium} and \eqref{SubEq:GradientErrorEquilibrium}. The KKT conditions are sufficient for optimality, thus $(\bar{x},\bar{u})$ must be a global minimizer of \eqref{Eq:OptimizationProblemFullInformation}. Therefore $\mathcal{E} \subset \chi$.
\end{proof}

\smallskip

Before moving on to the measurement case, let us make two comments regarding the matrix $\gradmat$.

First, the choice of $\gradmat$ is not unique; the only required property is \eqref{Eq:GradMatrix}. While a detailed study is outside the present scope, the choice of $\gradmat$ clearly affects the performance of the closed-loop system, and the sparsity patterns of different choices can provide flexibility for distributed implementations of OSS controllers. The construction of $\gradmat$ relies on knowledge of the plant matrices $A$ and $B$; this raises the question of what will happen if $A$ and $B$ are not known accurately. Suppose the actual plant matrices are $\tilde{A}$ and $\tilde{B}$. The controller \eqref{Eq:Controller} still guides the system to the optimizer of
\begin{align*}
  \minimize_{\bar{x},\bar{u}} &\quad f(\bar{x},\bar{u})\\
  \subto &\quad \tilde{A}\bar{x}+\tilde{B}\bar{u}+d=\vzeros[n]
\end{align*}
if $\nullspace\begin{bmatrix}A&B\end{bmatrix} = \nullspace\begin{bmatrix}\tilde{A}&\tilde{B}\end{bmatrix}$ and if the closed-loop system remains internally stable. We defer a more complete robustness analysis to a future paper.

Second, consider the advantage of enforcing the KKT conditions using the matrix $\gradmat$ over enforcing the KKT conditions using dual variables as employed in \cite{AJ-ML-PPJvdB:09b} or \cite{XZ-AP-NL:18} for example. The latter approach requires additional controller states representing the dual variables. Our alternative strategy omits such additional states, resulting in a lower-order controller.

\subsection{Controller Equations in the Measurement Case}\label{Sec:ControllerEquationsMeasurement}
We now return to the original input-output optimization problem \eqref{Eq:OptimizationProblem} with measurements $y = Cx$ and frame the measurement case as an instance of the full-state measurement case by defining $f(x,u) \coloneqq g(Cx,u)$. It is straightforward to prove from the definition of a convex function that $f$ is convex whenever $g$ is convex. We then apply the controller \eqref{Eq:Controller} with
\begin{equation*}
  \nabla f(x,u) = \begin{bmatrix}C^{\sf T}&\vzeros[]\\\vzeros[]&I_m\end{bmatrix}\nabla g(y,u)\,,
\end{equation*}
by the multidimensional chain rule. The closed-loop system becomes
\begin{subequations}\label{Eq:ControllerOutput}
  \begin{align}
    \dot{x} &= Ax+Bu+d\\
    y &= Cx\\
    e &= -R^{\sf T}\nabla g(y,u)\\
    \dot{\eta} &= e\\
    u &= K_I\eta + K_Pe\,, \label{SubEq:ControllerOutputPI}
  \end{align}
\end{subequations}
where
\begin{equation}\label{Eq:Py}
  R \define \begin{bmatrix}C&\vzeros[]\\\vzeros[]&I_m\end{bmatrix}\gradmat\,.
\end{equation}
Theorem \ref{Thm:EquilibriaAndOptimizers} holds for \eqref{Eq:ControllerOutput}, and under additional assumptions on the plant and the objective function $g$ we can further show uniqueness of the closed-loop equilibrium point.

\smallskip

\begin{proposition}[\bf Unique Equilibrium Point]\label{Prop:UniqueEquilibrium}
  Suppose the objective function $g$ is strictly convex and $(C,A)$ is detectable. Then for each $d \in \real^n$ the closed-loop system \eqref{Eq:ControllerOutput} has a unique equilibrium point.
\end{proposition}
\begin{proof}
  If $g$ is strictly convex, then for every optimizer $(x^\star,y^\star,u^\star)$ of \eqref{Eq:OptimizationProblem}, $y^\star$ and $u^\star$ are unique \cite[Section 4.2.1]{SB-LV:04}. If, additionally, $(C,A)$ is detectable then $\rank\begin{bmatrix}A\\C\end{bmatrix} = n$; it follows that the corresponding state $x^\star$ satisfying
  \begin{equation*}
    \begin{aligned}
      Ax^\star &= -Bu^\star-d\\
      Cx^\star &= y^\star
    \end{aligned}
  \end{equation*}
  is unique. Therefore, the set $\chi$ of \eqref{Eq:OptimizerSet} is a singleton. Since $\chi = \mathcal{E}$ by Theorem \ref{Thm:EquilibriaAndOptimizers}, we conclude \eqref{Eq:ControllerOutput} has a unique equilibrium point.
\end{proof}

\smallskip

We have so far shown that our controller solves the first sub-problem of the OSS control problem; namely, the equilibrium points of the closed-loop system are optimal. We turn our attention to the second sub-problem, that of stabilizing the closed-loop system.

\section{Stabilization}\label{Sec:Stabilization}
% The stability of the closed-loop system depends on the plant, the cost function, and the choice of the matrices $K_P$, $K_I$, and $\gradmat$. 
The proportional and integral gain matrices $K_P$ and $K_I$ may be tuned ``by hand,'' using a stability criterion to determine whether the arbitrarily chosen gains result in a stable system; alternatively, we might consider employing a synthesis procedure to design a stabilizing controller. The former is simpler, and we can continue to use the PI structure we have been discussing up to this point. The latter is more complex, requiring a \emph{dynamic stabilizer}, but systematizes gain tuning. We consider each strategy in turn.

\subsection{PI Controller Stability Analysis}
If we wish to select the gain matrices $K_P$ and $K_I$ heuristically, we require a method of assessing whether the resulting closed-loop system is stable. The only nonlinearity in the closed-loop system \eqref{Eq:ControllerOutput} is the memoryless operator $\nabla g$. If the assumptions of Proposition \ref{Prop:UniqueEquilibrium} hold, and a unique equilibrium point exists, then we can put the closed-loop system in exactly the form of the \emph{absolute stability problem} \cite[Section 7.1]{HKK:02} and apply the corresponding analysis tools. We therefore assume $(C,A)$ is detectable and $g$ is strictly convex in this section.

We centre the equations \eqref{Eq:ControllerOutput} about the unique (optimal) equilibrium point $(x^\star,\eta^\star)$  by defining the new state variables $\tilde{x} \define x-x^\star$ and $\tilde{\eta} \define \eta-\eta^\star$. We further define $y^\star \define Cx^\star$ and $u^\star \define K_I\eta^\star$. Straightforward calculations show that this change-of-variables results in the feedback interconnection of the memoryless operator
\begin{equation*}
  \Phi(\tilde{y},\tilde{u}) := \nabla g(\tilde{y}+y^\star,\tilde{u}+u^\star) - \nabla g(y^\star,u^\star)
\end{equation*}
and the LTI system $H$, which collects all the linear elements of \eqref{Eq:ControllerOutput}. The system $H$ has the realization
\begin{equation*}
  H = \left[
    \begin{array}{c|c}
      \mathcal{A}&\mathcal{B}\\\hline
      \mathcal{C}&\mathcal{D}
    \end{array}
  \right]
  \define
  \left[
    \begin{array}{c|c}
      \begin{array}{cc}
        A&BK_I\\\vzeros[]&\vzeros[]
      \end{array}&\begin{array}{cc}
                    BK_PR^{\sf T}\\R^{\sf T}
                  \end{array}\\\hline
      \begin{array}{cc}
        C&\vzeros[]\\\vzeros[]&K_I
      \end{array}&\begin{array}{c}
                    \vzeros[]\\K_PR^{\sf T}
                  \end{array}
    \end{array}
  \right]\,.
\end{equation*}

\begin{figure}[ht!]
\begin{center}
\tikzstyle{block} = [draw, fill=white, rectangle, 
    minimum height=3em, minimum width=6em, blur shadow={shadow blur steps=5}]
    \tikzstyle{hold} = [draw, fill=white, rectangle, 
    minimum height=3em, minimum width=3em, blur shadow={shadow blur steps=5}]
    \tikzstyle{dzblock} = [draw, fill=white, rectangle, minimum height=3em, minimum width=4em, blur shadow={shadow blur steps=5},
path picture = {
\draw[thin, black] ([yshift=-0.1cm]path picture bounding box.north) -- ([yshift=0.1cm]path picture bounding box.south);
\draw[thin, black] ([xshift=-0.1cm]path picture bounding box.east) -- ([xshift=0.1cm]path picture bounding box.west);
\draw[very thick, black] ([xshift=-0.5cm]path picture bounding box.east) -- ([xshift=0.5cm]path picture bounding box.west);
\draw[very thick, black] ([xshift=-0.5cm]path picture bounding box.east) -- ([xshift=-0.1cm, yshift=+0.4cm]path picture bounding box.east);
\draw[very thick, black] ([xshift=+0.5cm]path picture bounding box.west) -- ([xshift=+0.1cm, yshift=-0.4cm]path picture bounding box.west);
}]
\tikzstyle{sum} = [draw, fill=white, circle, node distance=1cm, blur shadow={shadow blur steps=8}]
\tikzstyle{input} = [coordinate]
\tikzstyle{output} = [coordinate]
\tikzstyle{pinstyle} = [pin edge={to-,thin,black}]

\begin{tikzpicture}[auto, scale = 0.6, node distance=2cm,>=latex', every node/.style={scale=1}]
  \node [input, name=input] {};
  \node [sum, right of=input] (sum) {};
  \node [hold, right of=sum, node distance=1.5cm] (system) {$H$};
  \node [hold, below of=system, node distance=1.5cm] (nonlin)  {$\Phi$};
  \node [output, right of=system, node distance=2cm] (output) {};
  \draw [draw,->] (input) -- node {$\vzeros[p+m]$} (sum);
  \draw [draw,->] (sum) -- (system);
  \draw [draw,->] (system) -- node [name=y] {$(y,u)$}(output);
  \draw [draw,->] (y) |- (nonlin);
  \draw [draw,->] (nonlin) -| node[pos=0.99] {$-$}
  node [near end] {} (sum);
\end{tikzpicture}
\end{center}
%\end{subfigure}
\caption{% The block diagram \ref{SubFig:CollectedSystem} shows the OSS controller with all the linear components placed in the forward path. 
  The feedback interconnection of Figure \ref{Fig:OssController} in the standard form of the absolute stability problem. The linear system $H$ collects all the linear elements and the nonlinearity $\Phi$ is a memoryless operator satisfying (incremental) sector bounds.}
\label{Fig:OssControllerStability}
\end{figure}
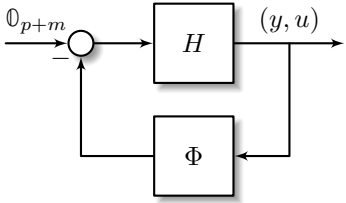

The operator $\Phi$ is a static nonlinearity in the sector $[0,\infty)$ owing to the convexity of $g$; if, furthermore, $g$ is strongly convex with parameter $\kappa$ and $\nabla g$ is Lipschitz continuous with parameter $L$, then $\Phi$ is in the sector $[\kappa,L]$.\footnote{Sector-boundedness is a property of the input-output behaviour of a static nonlinearity; see \cite[Definition 6.2]{HKK:02} for details.} The feedback interconnection of $H$ and $\Phi$, depicted in Figure \ref{Fig:OssControllerStability}, is therefore precisely in the form of the absolute stability problem. We present a stability criterion in the form of a \emph{linear matrix inequality} (LMI). For background on LMIs, see \cite{SB-LEG-EF-VB:94}. The following proposition is equivalent to \cite[Theorem 7.1]{HKK:02}, but has the advantage of being easy to verify with a program like \texttt{CVX} or \texttt{YALMIP}.

\smallskip

\begin{proposition}[\bf OSS Stability Criterion]\label{Prop:CircleCriterion}
  Suppose $g$ is strongly convex with parameter $\kappa > 0$ and $\nabla g$ is Lipschitz continuous with parameter $L > 0$. Let
  \begin{equation*}
    M \define \left(\begin{bmatrix}1&0\\0&-1\end{bmatrix}^{\sf T}\begin{bmatrix}-2\kappa L & \kappa+L\\\kappa+L & -2\end{bmatrix}\begin{bmatrix}1&0\\0&-1\end{bmatrix}\right) \otimes I_{p+m}\,,
  \end{equation*}
  where $\otimes$ denotes the Kronecker product. The unique equilibrium point of \eqref{Eq:ControllerOutput} is globally asymptotically stable if there exists a symmetric matrix $P = P^{\sf T} \in \real^{(n+m)\times(n+m)}$ and a real number $\alpha \geq 0$ such that
  \begin{equation*}
    \begin{bmatrix}I_{n+m}&\vzeros[]\\\mathcal{A}&\mathcal{B}\\\mathcal{C}&\mathcal{D}\\\vzeros[]&I_{p+m}\end{bmatrix}^{\sf T}\begin{bmatrix}\vzeros[]&P&\vzeros[]\\P&\vzeros[]&\vzeros[]\\\vzeros[]&\vzeros[]&\alpha M\end{bmatrix}\begin{bmatrix}I_{n+m}&\vzeros[]\\\mathcal{A}&\mathcal{B}\\\mathcal{C}&\mathcal{D}\\\vzeros[]&I_{p+m}\end{bmatrix} \prec \vzeros[]\,.
  \end{equation*}
\end{proposition}

\smallskip

\begin{proof}
 Apply \cite[Corollary 6]{JV-CWS-HK:16} to the system under consideration with the circle criterion multiplier \cite[Class 12]{JV-CWS-HK:16}.
\end{proof}

\smallskip

\begin{remark}[\bf Form of Matrix {$\boldsymbol M$}]\label{Rem:FormOfR}
  The term in parentheses in the definition of $M$ is a product of three matrices. The second matrix comes from the definition of a memoryless operator with a lower sector bound of $\kappa I_{p+m}$ and an upper sector bound of $L I_{p+m}$. The first and third matrices are the congruence transformation from \cite[Section 6.1]{MF-AR-MM-VMP:18} necessary to make the feedback interconnection positive instead of negative, as required by \cite[Corollary 6]{JV-CWS-HK:16}.\envend
\end{remark}

\smallskip

If the closed-loop system satisfies the LMI of Proposition \ref{Prop:CircleCriterion}, the OSS control problem is solved. Relaxing the strong convexity and Lipschitz continuity assumptions appears to be feasible but will necessitate a more sophisticated stability analysis, which we do not pursue here. We will demonstrate by simulation in Section \ref{Sec:IllustrativeExample} that the Lipschitz assumption is not necessary for stability.

Using the stability criterion of Proposition \ref{Prop:CircleCriterion}, one can attempt gain tuning by performing a search over the gain matrices $K_P$ and $K_I$ for selections that result in feasibility of the LMI. For instance, set $K_P = k_PI_m$ and $K_I = k_II_m$ with scalar parameters $k_P \in \real$ and $k_I \in \real$ and search over a grid for a stabilizing combination $(k_P,k_I) \in \real^2$. The drawback to this strategy is its heuristic nature.

\subsection{Dynamic Stabilizer Synthesis}
Should the ``hand-tuned'' method of the previous section fail, we can attempt a more general synthesis procedure at the cost of a more complex controller structure. We replace the PI gain equation \eqref{SubEq:ControllerOutputPI} with a dynamic stabilizer, so that the input $u$ is generated by
\begin{equation*}
  \begin{aligned}
    \dot{x}_{\rm s} &= A_{\rm s}x_{\rm s}+B_{\rm s}\sigma\\
    u &= C_{\rm s}x_{\rm s}+D_{\rm s}\sigma\,,
  \end{aligned}
\end{equation*}
in which $\sigma \define (y,\eta,e)$. The design variables are the matrices $A_{\rm s}$, $B_{\rm s}$, $C_{\rm s}$, and $D_{\rm s}$, which we select to \emph{enforce} that a stability LMI like the one of Proposition \ref{Prop:CircleCriterion} is satisfied.

The most straightforward synthesis method is to perform a \emph{loop transformation} to move the nonlinearity $\Phi$ from the sector $[\kappa,L]$ to the sector $[-1,1]$ and apply standard tools from the \emph{$\mathcal{H}_\infty$ synthesis problem} to minimize the $\mathscr{L}_2$ gain $\gamma$ of the augmented linear system. The resulting closed-loop system is stable by the small-gain theorem if $\gamma < 1$. For more on loop transformations, see \cite[Chapters 6, 7]{HKK:02}. For details on the $\mathcal{H}_\infty$ synthesis problem, see \cite[Chapter 3]{AI:17}.

\section{Illustrative Examples}\label{Sec:IllustrativeExample}
We illustrate the tools of Sections \ref{Sec:ControllerDesign} and \ref{Sec:Stabilization} through three simple examples. The first example shows the use of Proposition \ref{Prop:CircleCriterion} to verify stability using hand-tuned PI gains. The second example shows a simulation that suggests the Lipschitz bound on $\nabla g$ is not necessary for stability. The third example shows that a dynamic controller serves to stabilize the closed-loop system when the PI controller fails.

\subsection{Stability Verification Using Proposition \ref{Prop:CircleCriterion}}
We take as the plant matrices
\begin{align*}
  A &\define \begin{bmatrix}-1&-4&-1&3\\1&-4&-1&-3\\-1&4&-1&-9\\0&0&0&-4\end{bmatrix}\,, & B &\define \begin{bmatrix}0\\1\\0\\1\end{bmatrix}\,,\\
  C &\define \begin{bmatrix}1&-1&0&-4\\1&0&2&0\end{bmatrix}\,.
\end{align*}
One may verify that $(A,B)$ is stabilizable and $(C,A)$ is detectable. We make the columns of $\gradmat$ an orthonormal basis of $\nullspace \begin{bmatrix}A&B\end{bmatrix}$, resulting in
\begin{equation*}
  \gradmat^{\sf T} = \begin{bmatrix}0.1661&0.2491&-0.6644&0.1661&0.6644\end{bmatrix}\,.
\end{equation*}
We let $R$ be as defined in \eqref{Eq:Py}.

We suppose we are interested in stabilizing the system for \emph{any} objective function that is strongly convex with parameter $\kappa = 1/9$ and whose gradient is Lipschitz continuous with parameter $L = 1$. We were able to verify the corresponding optimizer is globally asymptotically stable using Proposition \ref{Prop:CircleCriterion} for the 100 gain combinations $(K_P,K_I) \in \{0.2,0.4,\ldots,2\}^2$.

\subsection{Non-Lipschitz Objective Function}\label{SubSec:NonLipschitz}
Although Proposition \ref{Prop:CircleCriterion} supposes $\nabla g$ is globally Lipschitz continuous, simulations show that the closed-loop system may be stable even if this assumption does not hold. Consider the objective function
\begin{equation*}
  g(y_1,y_2,u) = \cosh\left(\frac{y_1}{2}\right) + \cosh\left(\frac{y_2}{3}\right) + u^2\,.
\end{equation*}
The function $g$ is strongly convex with parameter $\kappa = 1/9$, but $\nabla g$ is not globally Lipschitz. We suppose the disturbance input $d(t)$ is $\begin{bmatrix}-1&3&1&2\end{bmatrix}^{\sf T}$ for  $0 \leq t < 5$, $\begin{bmatrix}2&-3&0&0\end{bmatrix}^{\sf T}$ for $5 \leq t < 10$, and $\begin{bmatrix}1&0&0&-1\end{bmatrix}^{\sf T}$ for $t \geq 10$. We set $K_I = 5$ and $K_P = 10$. Simulating the closed-loop system in \texttt{MATLAB} using the \texttt{ode15i} implicit ODE solver yields Figure \ref{Fig:ExampleSimulation}. The controller tracks the optimizer asymptotically for each value of the disturbance input, despite $\nabla g$ failing the global Lipschitz condition.

The matrix $M$ in Proposition \ref{Prop:CircleCriterion} can be modified to attempt stability verification in the case when $g$ is strongly convex but $\nabla g$ is not globally Lipschitz. Specifically, one may set
  \begin{equation*}
    M = \left(\begin{bmatrix}1&0\\0&-1\end{bmatrix}^{\sf T}\begin{bmatrix}-2\kappa & 1\\1 & 0\end{bmatrix}\begin{bmatrix}1&0\\0&-1\end{bmatrix}\right) \otimes I_{p+m}\,.
  \end{equation*}
  However, this modified LMI was infeasible for any gain value we tried. This suggests that in general, more sophisticated stability criteria are required when the objective function does not have a globally Lipschitz gradient. Note $\nabla g$ of this example is still \emph{locally} Lipschitz; that is, $\nabla g$ has a Lipschitz parameter when its domain is restricted to a bounded subset of $\real^{p+m}$. If one could establish boundedness of trajectories in advance, then semi-global stability results should follow.

\begin{figure}
  \begin{center}
    \includegraphics[width=\linewidth]{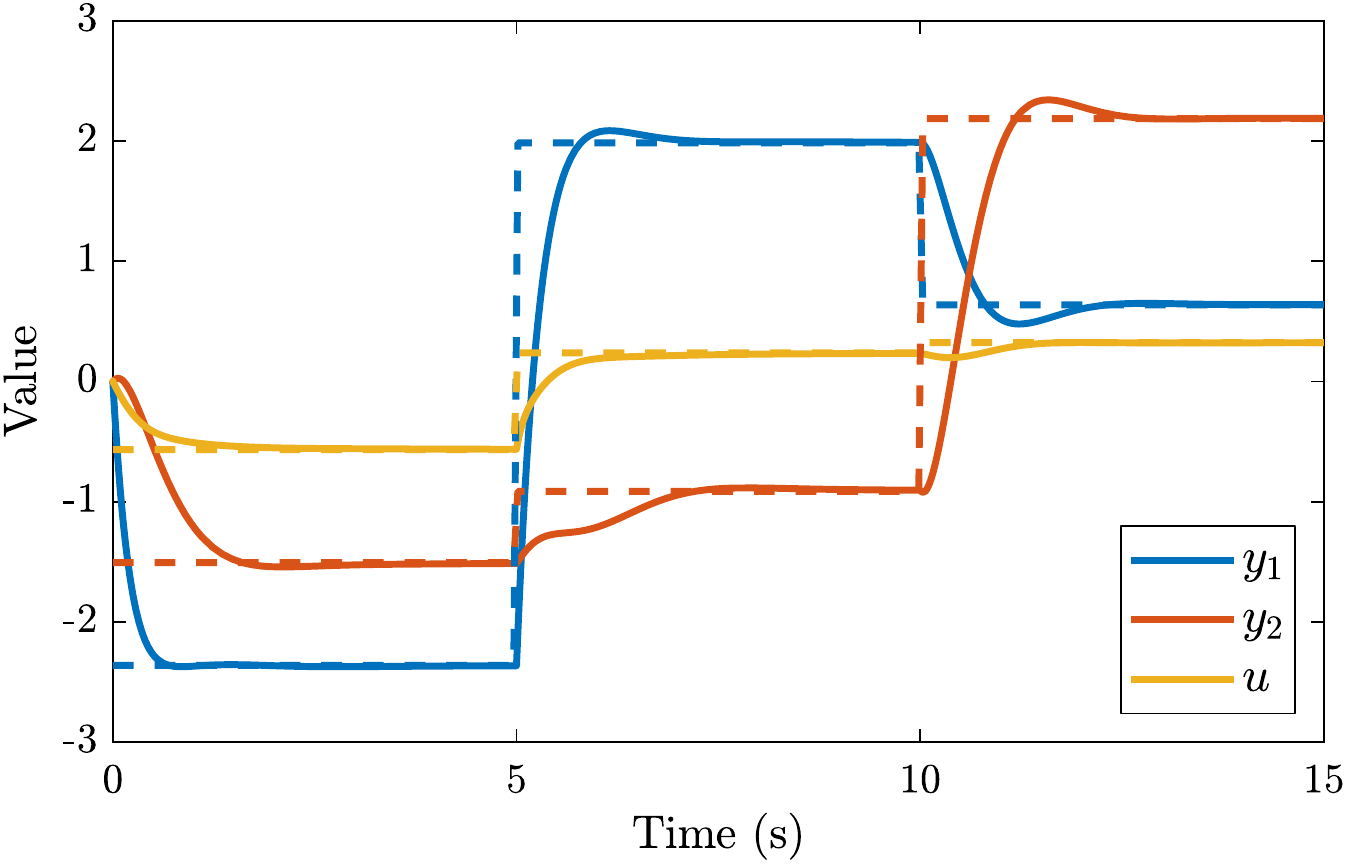}
  \end{center}
  \caption{The output variables $(y_1,y_2)$ and the input $u$ plotted as a function of time for the example of Section \ref{SubSec:NonLipschitz}. The optimizer for each variable over each time interval is shown as a dashed line.\label{Fig:ExampleSimulation}}
\end{figure}

\subsection{Dynamic Controller Synthesis for an Unstable Plant}\label{SubSec:DynamicController}
As mentioned in Section \ref{Sec:Stabilization}, an alternative to hand-tuning the PI controller and verifying stability using Proposition \ref{Prop:CircleCriterion} is to use a fully dynamic LTI controller and employ $\mathcal{H}_\infty$ synthesis methods to select the controller matrices. We suppose we are interested in OSS control for the \emph{unstable} plant
\begin{align*}
  A &\define \begin{bmatrix}-1&-4&-1&3\\1&-4&-1&-3\\-1&4&-1&-9\\0&0&0&1\end{bmatrix}\,, & B &\define \begin{bmatrix}0\\1\\0\\1\end{bmatrix}\,,\\
  C &\define \begin{bmatrix}1&-1&0&-4\\1&0&2&0\end{bmatrix}\,,
\end{align*}
and objective functions that are strongly convex with parameter $\kappa = 1$ and have Lipschitz continuous gradients with parameter $L = 2$. Once again $(A,B)$ is stabilizable and $(C,A)$ is detectable. The matrix $A$ has eigenvalues $\{-2,-2+2i,-2-2i,1\}$. We let the disturbance $d(t)$ be the same as in Section \ref{SubSec:NonLipschitz}.

We attempted stability verification using Proposition \ref{Prop:CircleCriterion} for the 49 gain combinations $(K_P,K_I) \in \{10^{-3},10^{-2},\ldots,10^{2},10^{3}\}^2$. The LMI solver failed in each case, and simulations further suggest the PI controller is incapable of stabilizing the closed-loop system. By contrast, we were able to synthesize a functioning dynamic controller whose behaviour is demonstrated in Figure \ref{Fig:DynamicStabilizer} for the objective function $g(y,u) = y_1^2+\frac{1}{2}y_2^2+\frac{1}{2}u^2$. While the dynamic controller is stabilizing, the closed-loop performance is poor. The response exhibits large overshoot and slow convergence to the optimizer. A more sophisticated synthesis procedure using the methods described in \cite{CS-PG-MC:97} could design for stability and performance simultaneously, and is one focus of future work.

\begin{figure}
  \centering
  \includegraphics[width=\linewidth]{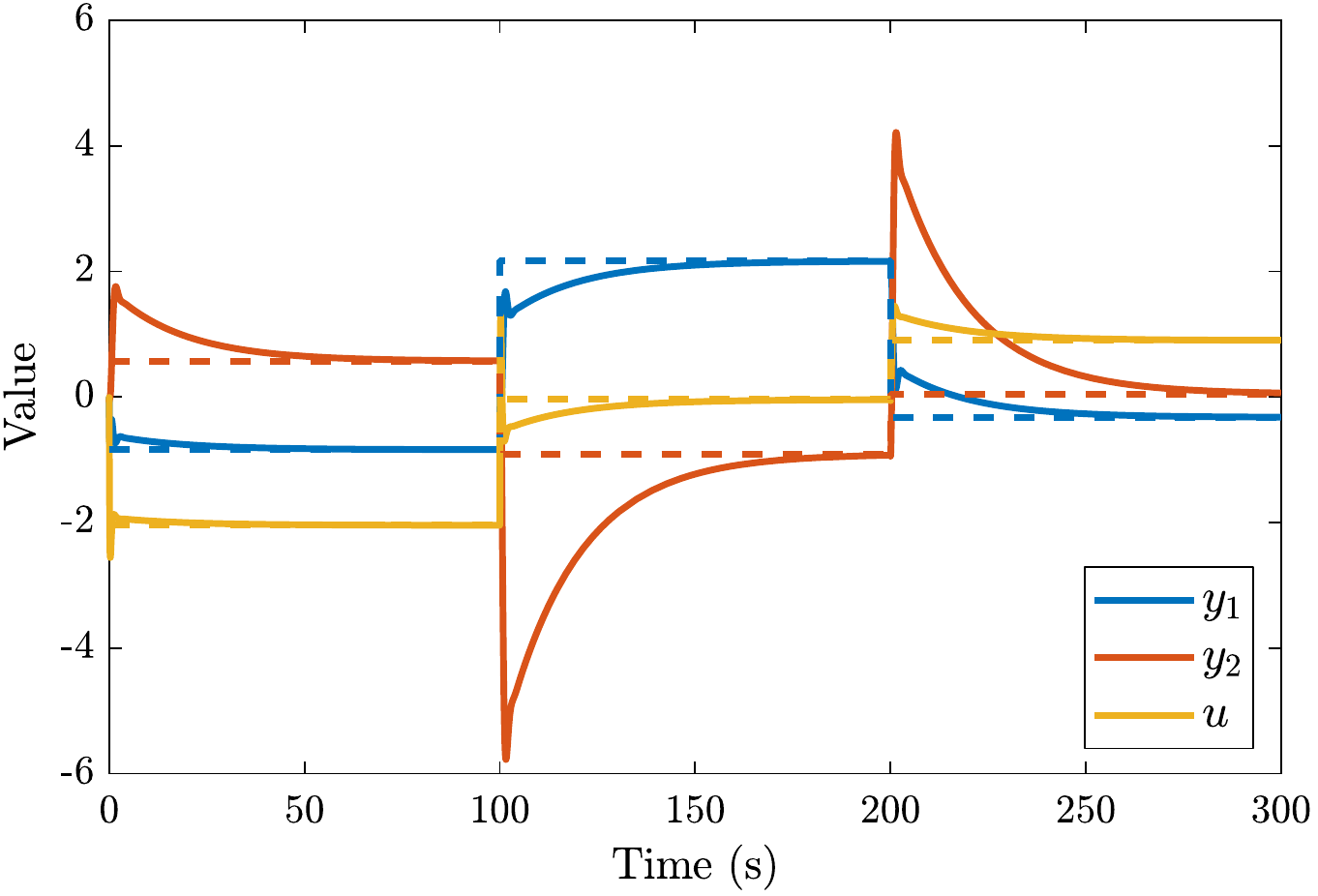}
  \caption{The output variables $(y_1,y_2)$ and the input $u$ plotted as a function of time for the example of Section \ref{SubSec:DynamicController}. The optimizer for each variable over each time interval is shown as a dashed line.}
\label{Fig:DynamicStabilizer}
\end{figure}

\section{Conclusions}\label{Sec:Conclusions}
We have presented a controller that solves the optimal steady-state control problem for linear time-invariant systems. Our controller is able to optimize over both the input and output by including the equilibrium equations as constraints of the optimization problem, but does not require additional dual variable controller states to be associated with these constraints. The potential application areas of optimal steady-state controllers are numerous \textemdash{} including power systems and chemical processing plants \textemdash{} as these controllers minimize the operating cost of any engineering system whose optimal set-point changes over time.

Future work will present a more general definition of the OSS control problem, with an exploration of the necessary and sufficient conditions on the controller. This generalization will include: adding equality and inequality constraints to the optimization problem to represent engineering constraints such as actuator limits or tie-line flow contracts in multi-area power systems; extending the approach to time-varying disturbances; explicitly considering robustness to parametric uncertainty; and outlining an architecture for OSS controllers that divides the controller structure into several sub-systems with well-defined roles.

%%%%%%%%%%%%%%%%%%%%%%%%%%%%%%%%%%%%%%%%%%%%%%%%%%%%%%%%%%%%%%%%%%%%%%%%%%%%%%%%
%%%%%%%%%%%%%%%%%%%%%%%%%%%%%%%%%%%%%%%%%%%%%%%%%%%%%%%%%%%%%%%%%%%%%%%%%%%%%%%%

\renewcommand{\baselinestretch}{1}
\bibliographystyle{IEEEtran}
\bibliography{alias,New,Main,JWSP}

%%%%%%%%%%%%%%%%%%%%%%%%%%%%%%%%%%%%%%%%%%%%%%%%%%%%%%%%%%%%%%%%%%%%%%%%%%%%%%%%
%%%%%%%%%%%%%%%%%%%%%%%%%%%%%%%%%%%%%%%%%%%%%%%%%%%%%%%%%%%%%%%%%%%%%%%%%%%%%%%%

\end{document}